\documentclass[12pt]{amsart}

\usepackage[utf8]{inputenc}
\usepackage[T1]{fontenc}
\usepackage{amsmath,amsfonts,amssymb,mathtools}
\usepackage{microtype}
\usepackage{xcolor}
\usepackage{hyperref}
\usepackage{enumitem}
\usepackage{biblatex}
\usepackage[margin=1in]{geometry}

\usepackage{tikz}
\usepackage[dvipsnames]{xcolor}

\DeclareFieldFormat[article]{title}{\mkbibemph{#1}}
\DeclareFieldFormat[article]{journaltitle}{#1}
\renewbibmacro{in:}{}

\newtheorem{theorem}{Theorem}[section]
\newtheorem{proposition}[theorem]{Proposition}
\newtheorem{lemma}[theorem]{Lemma}
\newtheorem{corollary}[theorem]{Corollary}

\theoremstyle{definition}
\newtheorem{definition}[theorem]{Definition}

\newtheorem{problem}[theorem]{Problem}
\theoremstyle{remark}
\newtheorem{remark}[theorem]{Remark}

\DeclareMathOperator{\conv}{conv}
\DeclareMathOperator{\supp}{supp}
\DeclareMathOperator{\relint}{relint}
\newcommand{\R}{\mathbb R}

\newcommand{\F}{\mathcal F}

\newcommand{\eps}{\varepsilon}

\title[Pushed Colorful Polytopal KKMS]{Pushed Colorful Polytopal KKMS Theorem and orthogonal star transversals}

\author[L. Mart\'inez-Sandoval]{Leonardo Mart\'inez-Sandoval}
\address[L. Mart\'inez-Sandoval]{Facultad de Ciencias, UNAM, Ciudad de M\'exico, M\'exico}
\email{leomtz@ciencias.unam.mx}

\author[A. Olvera-Sampieri]{Ariadna Olvera-Sampieri}
\address[A. Olvera-Sampieri]{Facultad de Ciencias, UNAM, Ciudad de M\'exico, M\'exico}
\email{arisam@ciencias.unam.mx}

\date{October 2026}

\begin{document}

\begin{abstract}

A classical theorem of Kaiser and Rabinovich implies that if every subfamily of size at most $\lfloor\log_2(d+1)\rfloor+1$ of a finite family of compact connected sets in $\mathbb{R}^d$ has a common axis-parallel hyperplane transversal, then the entire family has an orthogonal star transversal, i.e., a union of $d$ hyperplanes, one perpendicular to each coordinate axis. We study colorful variants of this result.

First, we use standard KKM machinery to establish a colorful analogue with $d+1$ color classes and a  $\lfloor\log_2(d+1)\rfloor+1$ transversal hypothesis for rainbow subfamiles. Then we focus on reducing the required number of colors. To this end, we introduce the rainbow pushability $\pi_r(P)$ of a polytope and we use it to state and prove few-color variants of the colorful Carathéodory theorem and the colorful polytopal KKMS theorem. Applying the latter to the $d$-dimensional cube, we show that  $\lfloor\log_2 d\rfloor+2$ color classes suffice for our transversal problem, provided that every rainbow choice has a common axis-parallel hyperplane transversal. We also extend both colorful results to the setting of convex $(n,D)$-bodies and we show that three colors are needed for $d\geq 2$.
\end{abstract}



\maketitle

\section{Introduction}

\subsection{Helly-type transversals, orthogonal stars and separated intervals}\label{subsec:orth-star-intro}

Helly's theorem is a cornerstone result in discrete geometry that turns a local condition intro a global one: for a finite family of convex sets in $\R^d$, there is a common point for the whole family exactly when each subfamily of size at most $d+1$ does. Its many extensions, known as \emph{Helly-type results}, establish some global property from local conditions involving only a bounded number of sets, known as the \emph{Helly number}.  Classical and modern surveys include those of Danzer, Gr\"unbaum, and Klee \cite{DanzerGrunbaumKlee1963}, Eckhoff \cite{Eckhoff1993}, Wenger \cite{Wenger2004}, Amenta, De Loera, and Sober\'on \cite{AmentaDeLoeraSoberon2017}, and Holmsen and Wenger \cite{HolmsenWenger2017}.

Here we study a Helly-type question in which the local and global properties differ. A hyperplane in $\R^d$ is \emph{axis-parallel} if it has the form
\[
 H_i(c):=\{x\in\R^d:x_i=c\}
\]
for some $i\in[d]$ and $c\in\R$.  Given $x=(x_1,\ldots,x_d)$, we write
\[
 S_x:=H_1(x_1)\cup\cdots\cup H_d(x_d)
\]
and call $S_x$ the \emph{orthogonal star centered at $x$}. We consider the following question.

\begin{problem}\label{prob:orth-star-helly}
Does there exist an integer $k=k(d)$ such that the following holds? \textit{If $\F$ is a finite family of compact connected subsets of $\R^d$ and every subfamily of size at most $k$ has a common axis-parallel hyperplane transversal, then the whole family admits an
orthogonal star transversal.} If so, how small can $k(d)$ be?
\end{problem}

This problem already has an affirmative answer in the literature, but recognizing it requires translating it into the language of separated $d$-intervals. We carefully document this equivalence, along with several others, in Proposition~\ref{prop:dictionary}. A \emph{separated $d$-interval} is a subset of the disjoint union of $d$ copies of $\R$, called \emph{tracks}, consisting of one compact interval on each track. It is \emph{full} if each interval is nonempty. A family of separated $d$-intervals is \emph{strongly intersecting} if there is a single point belonging to every member, and \emph{weakly intersecting} if there are $d$ points, one on each track, whose union meets every member. In these terms, Kaiser and Rabinovich \cite{KaiserRabinovich1999} in particular proved that if every subfamily of at most $k_d:=\left\lfloor\log_2(d+1)\right\rfloor+1$ full separated $d$-intervals is strongly intersecting, then the entire family is weakly intersecting i.e. Problem~\ref{prob:orth-star-helly} has an affirmative answer with $k(d)=k_d$.

The $d$-interval formulation also places the problem within the literature on transversal and covering properties of multiple intervals, initiated by Gy\'arf\'as and Lehel \cite{GyarfasLehel1985} and connected to axis-parallel rectangles by K\'arolyi and Tardos \cite{KarolyiTardos1996}. Classical contributions include piercing bounds by Tardos, Kaiser, and Alon \cite{Tardos1995,Kaiser1997,Alon1998}, and lower bounds by Matou\v{s}ek \cite{Matousek2001}. More recent developments include pairwise intersection results and connections to chordal edge-colorings by Bar\'at, Gy\'arf\'as, and S\'ark\"ozy \cite{BaratGyarfasSarkozy2025}, and a convexity approach by Rao \cite{Rao2025}. Our starting point is to seek colorful variants of Problem~\ref{prob:orth-star-helly}.

\subsection{Colorful orthogonal stars and two parameters}

Colorful Helly theory replaces one family by several color classes and checks only rainbow selections. The classic theorem is due to Lov\'asz, with the first published proof appearing in B\'ar\'any's work \cite{Barany1982}: if $\F_1,\ldots,\F_{d+1}$ are finite families of convex sets in $\R^d$ and for every choice $F_i\in \F_i$ ($i=1,\ldots,d+1$) we have that $F_1\cap \ldots \cap F_{d+1}$ is not empty, then one color class has nonempty intersection. The number $d+1$ of colors is optimal, although weaker transversal conclusions may survive with fewer colors \cite{MartinezRoldanRubin2020}.

Even when a colorful variant of a Helly-type theorem exists, the number of colors required may differ from the ordinary Helly number. For example, axis-parallel boxes in $\R^d$ have Helly number $2$, whereas their colorful counterpart requires $d+1$ color classes, although testing rainbow pairs still suffices to guarantee an intersecting color class. Related questions concerning the number of colors and the size of rainbow tests arise in work on Dol'nikov's problem concerning piercing translates of planar convex bodies \cite{JeronimoMagazinovSoberon2015,MartinezRoldan2024}, colorful line-transversal theorems \cite{McGinnisZerbib2022,Csizmadia2025}, and quantitative Helly theory \cite{SarkarXueSoberon2021}.

Motivated by these examples, we pose the following question

\begin{problem}\label{prob:colorful-orth-star}
For which pairs of integers $(m,k)$ does the following implication hold? If $\F_1,\ldots,\F_m$ are finite nonempty families of compact connected subsets of $\R^d$ and every rainbow subfamily of size at most $k$ has a common axis-parallel hyperplane transversal, must some color class have an orthogonal star transversal?
\end{problem}

We mostly focus on two natural parameters associated with this problem. Let $h_d$ be the smallest $k$ for which $(d+1,k)$ is a valid pair in Problem~\ref{prob:colorful-orth-star}, and let $c_d$ be the smallest $m$ for which $(m,m)$ is valid.

In Section~\ref{sec:benchmark}, we prove that $h_d\leq k_d$, where $k_d$ is the Kaiser--Rabinovich bound discussed above. This serves mainly as a benchmark, as our main transversal result provides a logarithmic upper bound for $c_d$, establishing a colorful result with substantially fewer colors.

\begin{theorem}\label{thm:c-main}
For every $d\ge1$,
\[
c_d\le\left\lfloor\log_2 d\right\rfloor+2.
\]
Equivalently, if $r=\left\lfloor\log_2d\right\rfloor+2$ and $\F_1,\ldots,\F_r$ are finite nonempty families of compact connected sets in $\R^d$ such that every complete rainbow choice admits a common axis-parallel hyperplane transversal, then some $\F_i$ admits an orthogonal star transversal.
\end{theorem}

For finite families, no compactness assumption is needed, and for compact connected sets a version with possibly infinite families holds as well (see Appendix~\ref{app:technical}). By the transversal dictionary of Proposition~\ref{prop:dictionary}, this theorem also admits a formulation in terms of full separated $d$-intervals, which we state in Corollary~\ref{cor:dintervals}. It is thus closely related to colorful piercing and matching results for $d$-intervals, including those of Frick and Zerbib \cite{FrickZerbib2019} and their sparse-colorful extensions by McGinnis and Zerbib \cite{McGinnisZerbib2024}.

While our benchmark follows from the standard colorful KKMS machinery described below, combined with the work of Kaiser--Rabinovich, reducing the number of colors requires a new approach. To this end, we introduce rainbow pushability and establish few-color variants of the colorful Carath\'eodory and polytopal KKMS theorems.

\subsection{KKM methods in transversal theory and the pushability trade-off}\label{subsec:kkm-intro}

The KKM theorem of Knaster, Kuratowski, and Mazurkiewicz \cite{KKM} has inspired a rich family of colorful and polytopal extensions, including those of Gale \cite{Gale1984}, Shapley \cite{Shapley1973}, Shih and Lee \cite{ShihLee1993}, and Komiya \cite{Komiya1994}. These developments culminated in the colorful polytopal KKMS theorem of Frick and Zerbib \cite{FrickZerbib2019}, which also contains colorful Carath\'eodory as a special case. More recently, sparse colorful KKM theorems have been developed in connection with fair division and other applications \cite{SoberonFair2022,McGinnisZerbib2024}.

KKM methods have also proved particularly useful in geometric transversal theory, including piercing problems for $d$-intervals \cite{AharoniKaiserZerbib2017,AharoniHolzmanZerbib2017,FrickZerbib2019}, the planar Hadwiger--Debrunner $(4,3)$ problem \cite{McGinnis43}, and Eckhoff's $T(3)$ problem \cite{McGinnisZerbib2022}. For a broader account of these applications, see the survey of McGinnis and Zerbib \cite{McGinnisZerbibSurvey}. The underlying strategy is to parameterize geometric configurations by a simplex or polytope, encode the relevant conditions through a KKM-type cover, and use topology to guarantee the existence of the desired configuration. Our approach follows this philosophy, using a cubical parameterization.

We record the colorful polytopal KKMS theorem of Frick and Zerbib below.

\begin{theorem}[Frick--Zerbib, colorful polytopal KKMS]\label{thm:FZ}
Let $P\subset\R^d$ be a $d$-dimensional polytope and let $p\in P$.  For each
$i\in[d+1]$, let
\[
 \{A_\sigma^i:\sigma\text{ a nonempty face of }P\}
\]
be an open or closed polytopal KKMS cover, meaning that for every nonempty face $G$ of
$P$,
\[
 G\subseteq\bigcup_{\varnothing\ne\sigma\subseteq G} A_\sigma^i.
\]
Prescribe a point $y_\sigma^i\in\sigma$ for every nonempty face $\sigma$ and every
$i\in[d+1]$.  Then there are nonempty faces $\sigma_1,\ldots,\sigma_{d+1}$ such that
\[
 p\in\conv\{y_{\sigma_1}^1,\ldots,y_{\sigma_{d+1}}^{d+1}\}
 \qquad\text{and}\qquad
 \bigcap_{i=1}^{d+1}A_{\sigma_i}^i\ne\varnothing.
\]
\end{theorem}
 
The dimension enters here twice: $P$ is $d$-dimensional and the theorem uses $d+1$ covers, which typically determine the number of colors in transversal applications. However, as our benchmark proof in Section~\ref{sec:benchmark} will reveal, the convex-hull conclusion is stronger than we need: it suffices that the selected points are not all contained in a common facet of $P$. This observation will allow us to reduce the number of covers, and hence the number of colors.

This motivates the following definition. For a centered polytope $P$ (i.e., $0\in\operatorname{int}P$), let $\pi_r(P)$ be the maximum, over all choices of $r$ independent random vertices of $P$ whose expectations sum to zero, of the probability that they lie in a common facet. We call this the $r$-\emph{rainbow pushability} of $P$ (see Definition~\ref{def:pushability}). The following two theorems show how this parameter allows us to weaken the usual colorful convex hull conclusion.

First, restricting the definition of rainbow pushability to vertex valued random variables is the same as if the variables are allowed to take arbitrary values in $\partial P$ or even in $P$. This leads to the following relaxation of the classical colorful Carath\'eodory theorem \cite{Barany1982}.

\begin{theorem}[Pushed colorful Carath\'eodory]\label{thm:pushed-caratheodory}
Let $P\subset\R^d$ be a centered polytope with $\pi_r(P)<1$. If $X_1,\ldots,X_r\subseteq\partial P$ are nonempty sets satisfying
\[
0\in\conv X_i\qquad(i\in[r]),
\]
then there are $x_i\in X_i$ such that
\[
\relint\conv\{x_1,\ldots,x_r\}\subseteq\operatorname{int}P.
\]
Equivalently, the selected points are not contained in any common proper face of $P$.
\end{theorem}

The classical version requires $d+1$ colors in general to guarantee that the convex hull of a rainbow selection contains the origin. Our theorem allows potentially fewer colors by relaxing this conclusion: instead of containing the origin, the rainbow convex hull need only avoid being contained in a facet of $P$.

Second, the same pushability threshold yields a relaxation of the polytopal KKMS theorem. The number $r$ of covers below is not fixed
by $\dim P$ if we allow the conclusion to be relaxed.

\begin{theorem}[Pushed Colorful Polytopal KKMS]\label{thm:pushed-kkms}
Let $P$ be a polytope with $0\in\operatorname{int}P$, and let $r\ge1$.  Suppose
$\pi_r(P)<1$.  For each $i\in[r]$, let
\[
 \{A_\sigma^i:\sigma\text{ a nonempty face of }P\}
\]
be a closed polytopal KKMS cover of $P$, and prescribe a point $y_\sigma^i\in\sigma$
for every nonempty face $\sigma$.  Then there are nonempty faces
$\sigma_1,\ldots,\sigma_r$ such that
\[
\relint\conv\{y_{\sigma_1}^1,\ldots,y_{\sigma_r}^r\}
 \subseteq\operatorname{int}P  \qquad \text{and} \qquad \bigcap_{i=1}^r A_{\sigma_i}^i\ne\varnothing. 
\]
In particular, the selected faces are not all contained in a common facet of $P$.  The
same conclusion holds for open KKMS covers.
\end{theorem}

The cube provides a concrete illustration of this trade-off.

\begin{theorem}\label{thm:cube-push}
For $C_d=[-1,1]^d$,
\[
 \pi_r(C_d)=\min\left\{1,\frac{d}{2^{r-1}}\right\}.
\]
\end{theorem}

Thus $\pi_r(C_d)<1$ precisely when $r\geq \left\lfloor\log_2 d\right\rfloor+2$, and as we will see, 
Theorem~\ref{thm:c-main} will follow from Theorem~\ref{thm:pushed-kkms}, Theorem \ref{thm:cube-push} and a \textit{cubical pushing lemma}, Lemma \ref{lem:pushing}.

\subsection{Lower bounds and further polytopes}

Our logarithmic upper bounds are still far from the known lower bounds. In Proposition~\ref{prop:pairwise-obstruction}, we construct, for every $d\ge3$, a family of $2d$ axis-parallel boxes such that every pair has a common axis-parallel hyperplane transversal, but no orthogonal star meets the entire family. This gives $h_d,c_d\ge3$. Our construction arises from an explicit decomposition of $K_{2d}$ into $d$ interval graphs (Lemma~\ref{lem:ariadna-decomp}) and improves the size of an earlier construction by Bar\'at, Gy\'arf\'as, and S\'ark\"ozy \cite{BaratGyarfasSarkozy2025}. Together with a separate planar obstruction (Proposition~\ref{prop:two-colors-plane}) and our upper bounds, we obtain $h_3=h_4=h_5=h_6=3$ and $c_2=c_3=3$.

Our pushed theorems also apply beyond cubes. In Theorem~\ref{thm:facet-depth}, we bound $\pi_r(P)$ for any centered polytope in terms of its facet depths relative to the distinguished center. In particular, Corollary~\ref{cor:cs-push} shows that if $P=-P$, then
\[
\pi_r(P)\le\min\left\{1,\frac{f_{d-1}(P)}{2^r}\right\},
\]
where $f_{d-1}(P)$ denotes the number of facets.

\subsection{The full Kaiser--Rabinovich setting}

Kaiser and Rabinovich work more generally with $(n,D)$-bodies, ordered $n$-tuples of
convex subsets of $\R^D$: strong intersection means a common point in one component,
whereas weak intersection means piercing by one point in each component space
\cite{KaiserRabinovich1999}. In this setting, we establish two possible colorful versions. Put
\[
 L_n:=\left\lfloor\log_2(n+1)\right\rfloor+1,
 \qquad
 R_n:=\left\lfloor\log_2n\right\rfloor+2.
\]

\begin{theorem}\label{thm:KR-tradeoffs}
Let $n,D\ge1$.
\begin{enumerate}[label=(\roman*)]
\item With $(n+1)^D$ color classes of convex $(n,D)$-bodies, if every rainbow subfamily
of size at most $L_n^D$ is strongly intersecting, then one color class is weakly
intersecting.
\item With $R_n^D$ color classes of convex $(n,D)$-bodies, if every complete rainbow
choice is strongly intersecting, then one color class is weakly intersecting.
\end{enumerate}
\end{theorem}

The first statement has exactly the same local test size $L_n^D$ as the original
Kaiser--Rabinovich upper bound for the monochromatic Helly-like number.  The second is the
higher-dimensional descendant of Theorem~\ref{thm:c-main}: the number of external colors
is itself only polylogarithmic in $n$ for fixed $D$.  Section~\ref{sec:actual-KR} proves
both statements from a colorful form of Kaiser--Rabinovich dimension submultiplicativity.

\subsection{Organization of the paper}

Section~\ref{sec:benchmark} develops the cubical encoding and the Frick--Zerbib/
Kaiser--Rabinovich benchmark.  Sections~\ref{sec:pushability} and \ref{sec:computing}
develop rainbow pushability and prove the main few-color theorem.  The remaining sections
treat separated $d$-intervals and lower constructions, the full Kaiser--Rabinovich setting,
and open problems. We defer technical topological and compactness details to the appendices.

\section{The cubical benchmark: encoding and pushing}\label{sec:benchmark}

We first record the elementary dictionary that will be used repeatedly below.

\begin{proposition}[Transversal dictionary]\label{prop:dictionary}
Let $\F=\{F_j:j\in J\}$ be a finite family of nonempty connected subsets of $\R^d$.
For $i\in[d]$ and $j\in J$, put $I_i^j$ as projection of $F_j$ to the $i$-th coordinate and
\[
 B_j:=I_1^j\times\cdots\times I_d^j,
\]
and let $G_i$ be the interval graph on vertex set $J$ represented by the intervals
$\{I_i^j:j\in J\}$.

For every $S\subseteq J$, the following are equivalent:
\begin{enumerate}[label=(\alph*)]
\item $\{F_j:j\in S\}$ has a common axis-parallel hyperplane transversal;
\item $\{B_j:j\in S\}$ has a common axis-parallel hyperplane transversal;
\item for some $i\in[d]$, $\bigcap_{j\in S} I_i^j\ne\varnothing$;
\item $S$ is a clique of $G_i$ for some $i\in[d]$.
\end{enumerate}
Moreover, the following are equivalent:
\begin{enumerate}[label=(\alph*)]
\item $\F$ has an orthogonal star transversal;
\item $\{B_j:j\in J\}$ has an orthogonal star transversal;
\item there are points $x_1,\ldots,x_d\in\R$ such that
\[
 J=\bigcup_{i=1}^d\{j\in J:x_i\in I_i^j\};
\]
\item there are possibly empty cliques $Q_i$ of $G_i$, $i\in[d]$, with
\[
 J=Q_1\cup\cdots\cup Q_d.
\]
\end{enumerate}
\end{proposition}

The proof is elementary and is included in Appendix~\ref{app:dictionary}.  In particular,
both transversal notions depend only on the coordinate boxes $B_j$.

We now establish $h_d\leq k_d$. In this setting, we are using $d+1$ colors. The proof is short once the geometry of orthogonal stars is encoded in the cube, but the encoding also exposes the ``pushing'' phenomenon that will motivate the general parameter of the next section. Since the families are finite and compact, after replacing each member by its coordinate box and rescaling we may assume that every set under consideration is contained in $C_d=[-1,1]^d$.

For $x=(x_1,\ldots,x_d)\in C_d$ and $v=(v_1,\ldots,v_d)\in V(C_d)=\{-1,1\}^d$,
define
\[
 R_x^v:=\{y\in C_d:(x_i-y_i)v_i>0\text{ for every }i\in[d]\}.
\]
Thus $R_x^v$ is the open orthant of $C_d\setminus S_x$ opposite to the vertex $v$.
If $x$ lies on a face $G$ of $C_d$ and $v\notin V(G)$, then $R_x^v=\varnothing$.

For a color class $\F_j$ and $v\in V(C_d)$ set
\[
 A_v^j:=\{x\in C_d:\text{there exists }F\in\F_j\text{ with }F\subset R_x^v\}.
\tag{2.1}
\]
These sets are open.  If for some $x$ we have $x\notin\bigcup_vA_v^j$, then $S_x$ is
already transversal to $\F_j$: otherwise a set $F\in\F_j$ disjoint from $S_x$ has every
coordinate projection entirely on one side of $x_i$, and hence lies in some $R_x^v$.
Therefore, if $\F_j$ has no orthogonal star transversal, then
\[
 C_d=\bigcup_{v\in V(C_d)}A_v^j.
\]
Moreover,
\[
 G\subseteq\bigcup_{v\in V(G)}A_v^j
\]
for every face $G$.  Thus every color class without an orthogonal star transversal
produces a cubical KKM cover.

The next elementary observation is the bridge between the local transversal hypothesis
and the combinatorics of cube facets.

\begin{lemma}[Cubical pushing lemma]\label{lem:pushing}
Let $x\in C_d$, let $v_1,\ldots,v_m\in V(C_d)$, and let
$F_i\subset R_x^{v_i}$ be nonempty.  If one axis-parallel hyperplane is transversal to
$F_1,\ldots,F_m$, then $v_1,\ldots,v_m$ lie on a common facet of $C_d$.
\end{lemma}

\begin{proof}
Let the transversal be $H_k(a)=\{y:y_k=a\}$.  Choose $p_i\in F_i\cap H_k(a)$.  Since
$p_i\in R_x^{v_i}$, the $k$-th coordinates $v_{i,k}$ and $p_{i,k}$ of $v_i$ and $p_i$ satisfy
\[
 (x_k-a)v_{i,k}=(x_k-p_{i,k})v_{i,k}>0
\]
for every $i$.  Hence all $v_{i,k}$ have the same sign, so all $v_i$ lie in the facet
$z_k=1$ or all lie in the facet $z_k=-1$.
\end{proof}

We now combine this with two additional ingredients. The first is
Lemma~4.3 of Kaiser and Rabinovich \cite{KaiserRabinovich1999}.

\begin{lemma}\label{lem:KR}
Let $S\subseteq\{-1,1\}^d$ satisfy $0\in\conv S$.  Then there is a subset $T\subseteq S$
with
\[
 |T|\le\left\lfloor\log_2(d+1)\right\rfloor+1
\]
which is not contained in any facet of $C_d$.  Equivalently, for every coordinate, the
vectors of $T$ exhibit both signs.
\end{lemma}

The second is the Theorem~\ref{thm:FZ}, specialized to vertex-indexed
covers of the cube. We can now prove our benchmark result.

\begin{theorem}\label{thm:h-benchmark}
For every $d\ge1$,
\[
 h_d\le\left\lfloor\log_2(d+1)\right\rfloor+1.
\]
Equivalently, let $k_d=\lfloor\log_2(d+1)\rfloor+1$.  If
$\F_1,\ldots,\F_{d+1}$ are finite nonempty families of compact connected subsets of
$\R^d$ such that every rainbow subfamily of size at most $k_d$ has a common
axis-parallel hyperplane transversal, then some color class has an orthogonal star
transversal.
\end{theorem}

\begin{proof}
Assume no color class has an orthogonal star transversal.  The $d+1$ families therefore produce the open vertex-indexed KKM covers (2.1).  To apply Theorem~\ref{thm:FZ}, set $A_\sigma^j=\varnothing$ for every nonvertex face $\sigma$, take $p=0$, and prescribe $y_v^j=v$ for vertices $v$ (the representatives of nonvertex faces may be chosen arbitrarily).  Since the selected sets must have nonempty common intersection, the theorem selects only vertices.  We therefore obtain a common point $x$, labels $v_1,\ldots,v_{d+1}$ with $0\in\conv\{v_1,\ldots,v_{d+1}\}$, and witnesses $F_j\in\F_j$ with $F_j\subset R_x^{v_j}$.  Discard repeated vertices and apply Lemma~\ref{lem:KR}.  It gives at most $\lfloor\log_2(d+1)\rfloor+1$ distinct vertices which are not contained in a common facet. The corresponding witnesses form a rainbow subfamily.  By hypothesis they have a common axis-parallel hyperplane transversal, so Lemma~\ref{lem:pushing} puts their labels in a common facet, a contradiction.
\end{proof}

This proof reveals that the geometric application ultimately needs only
a rainbow selection escaping every common facet.

\section{Rainbow pushability and pushed colorful convexity}\label{sec:pushability}

\subsection{The parameter}

A \emph{centered polytope} is a polytope $P\subset\R^d$ with $0\in\operatorname{int}P$.  We write $V(P)$ for its vertex set.

For probability measures $\mu_1,\ldots,\mu_r$ on $V(P)$, let $V_1,\ldots,V_r$ be independent random vertices with $V_i\sim\mu_i$.  We call the measures \emph{collectively balanced} if
\[
 \sum_{i=1}^r \mathbb E V_i=0.
\tag{3.1}
\]
Such tuples of measures always exist: since $0\in\conv V(P)$, take the same zero-mean measure in every color.

\begin{definition}[Rainbow pushability]\label{def:pushability}
The \emph{$r$-rainbow pushability} of a centered polytope $P$ is
\[
 \pi_r(P):=
 \max_{\substack{\mu_1,\ldots,\mu_r\text{ on }V(P)\\
                  \sum_i\mathbb EV_i=0}}
 \Pr\bigl(V_1,\ldots,V_r\text{ lie in a common facet of }P\bigr).
\tag{3.2}
\]
\end{definition}

Although the definition is phrased probabilistically, the underlying selection problem is entirely deterministic: the measures are simply convex weights on the vertices.  The probabilistic language is just useful, as it allows both structural properties and the calculation of the parameter to be expressed cleanly.

The maximum exists because the feasible set is a compact polytope in a finite product of probability simplices and the objective is continuous.  The parameter is invariant under invertible linear maps fixing the origin: if $T\in\mathrm{GL}(d)$, then $\pi_r(TP)=\pi_r(P)$.  The meaning of $\pi_r(P)<1$ is simple: every collectively balanced system of $r$ independent vertex distributions admits at least one rainbow outcome that escapes all facets.

\subsection{Equivalent definitions and pushed colorful Carath\'eodory}

Although Definition~\ref{def:pushability} is finite and vertex valued, the same number is associated to distributions on the boundary and on the whole polytope.  We use only finitely supported distributions below, as this is sufficient for all applications in this paper.

\begin{proposition}\label{prop:vertex-reduction}
Define $\pi_r^\partial(P)$ and $\pi_r^{\mathrm{poly}}(P)$ by the analogous suprema,
but allowing the independent random variables to be finitely supported on
$\partial P$ and on $P$, respectively.  Then
\[
 \pi_r(P)=\pi_r^\partial(P)=\pi_r^{\mathrm{poly}}(P).
\]
\end{proposition}

\begin{proof}
The inequalities
\[
 \pi_r(P)\le\pi_r^\partial(P)\le\pi_r^{\mathrm{poly}}(P)
\]
are immediate.  For the reverse inequality, let $X_1,\ldots,X_r$ be independent,
finitely supported $P$-valued random variables with $\sum_i\mathbb EX_i=0$.  For every
point $x$ in the support of $X_i$, let $G_x=\supp(x)$ be the minimal face of $P$
containing $x$.  Since $x\in\relint G_x$, choose a convex representation
\[
 x=\sum_{v\in V(G_x)}\lambda_v(x)v
\]
with positive coefficients.  Conditional on $X_i=x$, replace $x$ by a random vertex
$\widehat X_i\in V(G_x)$ with probabilities $\lambda_v(x)$.  Then
\[
 \mathbb E(\widehat X_i\mid X_i=x)=x,
 \qquad\text{hence}\qquad
 \mathbb E\widehat X_i=\mathbb EX_i.
\]
Thus the refined vertex distributions remain collectively balanced.

If $x$ lies in a facet $F$, minimality of $G_x$ gives $G_x\subseteq F$.  Hence every
possible vertex refinement of $x$ also lies in $F$.  Consequently, whenever
$X_1,\ldots,X_r$ lie in a common facet, their vertex refinements do so as well.  Thus
refinement cannot decrease the probability of lying in a common facet, while the refined
vertex distributions have probability at most $\pi_r(P)$.  Taking the supremum proves
$\pi_r^{\mathrm{poly}}(P)\le\pi_r(P)$.
\end{proof}

The geometric meaning of escaping all facets is also simple.

\begin{lemma}\label{lem:relint-criterion}
For points $x_1,\ldots,x_r\in P$, the following are equivalent:
\begin{enumerate}[label=(\roman*)]
\item $x_1,\ldots,x_r$ are not contained in a common proper face of $P$;
\item $x_1,\ldots,x_r$ are not contained in a common facet of $P$;
\item
\[
 \relint\conv\{x_1,\ldots,x_r\}\subseteq\operatorname{int}P.
\]
\end{enumerate}
\end{lemma}

\begin{proof}
Every proper face lies in a facet, so (i) and (ii) are equivalent.  If all selected points
lie in a proper face, their convex hull lies in $\partial P$, so (iii) fails.  Conversely,
suppose a point $z$ of the relative interior of their convex hull lies on $\partial P$.
Choose a supporting functional $\ell\le1$ on $P$ with $\ell(z)=1$.  If some other point $z'$ of the selected convex hull satisfied $\ell(z')<1$, then,
since $z$ lies in the relative interior, the ray from $z'$ through $z$ could be continued
slightly past $z$ while remaining in the convex hull.  At such a point $z''$ we would
have $\ell(z'')>1$, contradicting $\ell\le1$ on $P$.  Hence $\ell=1$ throughout the
convex hull, which therefore lies in the proper face $P\cap\{\ell=1\}$.
\end{proof}

\begin{proposition}\label{prop:pushed-selection}
Let $P\subset\R^d$ be centered and suppose $\pi_r(P)<1$.  Let
$X_1,\ldots,X_r\subseteq P$ be nonempty sets.  If there exist points
\[
 b_i\in\conv X_i\qquad(i\in[r])
\]
such that
\[
 b_1+\cdots+b_r=0,
\]
then there are $x_i\in X_i$ which are not contained in a common facet of $P$.
Equivalently,
\[
 \relint\conv\{x_1,\ldots,x_r\}\subseteq\operatorname{int}P.
\]
\end{proposition}

\begin{proof}
For each $i$, choose a finite convex representation of $b_i$ using points of $X_i$ and
regard its coefficients as a finitely supported probability distribution on $X_i$.
Choose the resulting random variables independently.  They are collectively balanced.
By Proposition~\ref{prop:vertex-reduction}, the probability that they lie in a common facet is at most
$\pi_r(P)<1$.  Hence some rainbow outcome $x_i\in X_i$ is not contained in a common facet.  The final
formulation follows from Lemma~\ref{lem:relint-criterion}.
\end{proof}

\begin{proof}[Proof of Theorem~\ref{thm:pushed-caratheodory}]
Apply Proposition~\ref{prop:pushed-selection} with $b_i=0$ for every $i$; these choices
are available because $0\in\conv X_i$.  Lemma~\ref{lem:relint-criterion} gives the
stated interior-penetration conclusion.
\end{proof}

\subsection{The pushed colorful polytopal KKMS theorem}

The following is the vertexwise form of the Sperner--Shapley step used by Frick and
Zerbib \cite[proof of Theorem~2.2]{FrickZerbib2019}.  In their statement all vertices
with the same face label use the same fixed point. Here the point may vary from vertex to
vertex.  Their face-preserving degree argument applies essentially unchanged, see Appendix~\ref{app:technical} for details.

\begin{lemma}\label{lem:face-preserving-selection}
Let $T$ be a triangulation of a polytope $P$, and assign to every vertex $u\in V(T)$ a
point $z(u)$ in $\supp(u)$. Then for every $p\in P$
there is a simplex $\tau\in T$ such that
\[
 p\in\conv\{z(u):u\in V(\tau)\}.
\]
\end{lemma}

This is the same input as in the colorful KKMS polytopal theorem of Frick and Zerbib. The new ingredient below is what we do with it: Proposition~\ref{prop:pushed-selection} converts collective balance into a
rainbow choice that escapes every common facet.

\begin{proof}[Proof of Theorem~\ref{thm:pushed-kkms}]
We first assume the covers are closed.  Fix $\eps>0$ and triangulate $P$ so that every
simplex has diameter at most $\eps$.  For every vertex $u$ of the triangulation and every
color $i\in[r]$, the KKMS condition applied to $\supp(u)$ allows us to choose a nonempty face
\[
 \sigma_i(u)\subseteq\supp(u)
 \qquad\text{with}\qquad
 u\in A_{\sigma_i(u)}^i.
\tag{3.3}
\]
Use the prescribed representative $y_{\sigma_i(u)}^i\in\sigma_i(u)$ and define
\[
 z(u):=\frac1r\sum_{i=1}^r y_{\sigma_i(u)}^i.
\]
Since $\supp(u)$ is convex and every $y_{\sigma_i(u)}^i$ belongs to $\supp(u)$, we have
$z(u)\in\supp(u)$.  Apply Lemma~\ref{lem:face-preserving-selection} with $p=0$.  It gives a simplex
$\tau$ of the triangulation such that
\[
 0\in\conv\{z(u):u\in V(\tau)\}.
\]
Choose convex coefficients $\lambda_u$, $u\in V(\tau)$, witnessing this membership.
Then
\[
 0=\sum_u\lambda_u z(u)
   =\frac1r\sum_{i=1}^r\sum_u\lambda_u y_{\sigma_i(u)}^i.
\tag{3.4}
\]

For each color define
\[
 X_i^\tau:=\{y_{\sigma_i(u)}^i:u\in V(\tau)\}\subseteq P,
 \qquad
 b_i:=\sum_{u\in V(\tau)}\lambda_u y_{\sigma_i(u)}^i\in\conv X_i^\tau.
\]
Equation~(3.4) says that
\[
 b_1+\cdots+b_r=0.
\]
Applying Proposition~\ref{prop:pushed-selection} to $X_1^\tau,\ldots,X_r^\tau$, choose
vertices $u_i\in V(\tau)$ so that the prescribed points
$y_{\sigma_i(u_i)}^i$, $i\in[r]$, are not contained in a common facet of $P$.
By Lemma~\ref{lem:relint-criterion},
\[
 \relint\conv\{y_{\sigma_1(u_1)}^1,\ldots,y_{\sigma_r(u_r)}^r\}
 \subseteq\operatorname{int}P.
\]
On the other hand, $u_i\in A_{\sigma_i(u_i)}^i$ by (3.3), and all $u_i$ belong to the
same simplex $\tau$.  Since $\operatorname{diam}(\tau)\le\eps$, any one of the $u_i$
lies within $\eps$ of every selected set.  Thus the $\eps$-neighborhoods of
\[
 A_{\sigma_1(u_1)}^1,\ldots,A_{\sigma_r(u_r)}^r
\]
have a common point.

Let $\eps_n\to0$.  There are only finitely many $r$-tuples of faces of $P$, so after
passing to a subsequence the same tuple $\sigma_1,\ldots,\sigma_r$ occurs for every $n$.
The prescribed representatives are then fixed and still satisfy
\[
 \relint\conv\{y_{\sigma_1}^1,\ldots,y_{\sigma_r}^r\}
 \subseteq\operatorname{int}P.
\]
Choose $x_n\in P$ in the common $\eps_n$-neighborhood of the selected sets.  By
compactness, a subsequence converges to some $x\in P$.  Since
$\operatorname{dist}(x_n,A_{\sigma_i}^i)\le\eps_n$ and each $A_{\sigma_i}^i$ is closed,
$x\in A_{\sigma_i}^i$ for every $i$.  Hence
\[
 \bigcap_{i=1}^r A_{\sigma_i}^i\ne\varnothing.
\]
This proves the closed case.

For open KKMS covers one uses the standard shrinking argument: since the face set is
finite, the open sets can be replaced by closed subsets preserving all KKMS covering
conditions. See, for example, the proof of Komiya's theorem in
\cite{FrickZerbib2019}.  Applying the closed case to the shrunk cover gives the result.
\end{proof}

\section{Computing rainbow pushability and the main transversal theorem}\label{sec:computing}

\subsection{A general facet-depth bound}

For a facet $F$ of $P$, choose the supporting functional $\ell_F$ normalized by
\[
 F=P\cap\{\ell_F=1\},\qquad \ell_F\le1\text{ on }P,
\]
and put
\[
 a_F:=-\min_{x\in P}\ell_F(x)>0,
 \qquad
 \theta_F:=\frac{a_F}{1+a_F}.
\]
The number $\theta_F$ measures the depth of the origin relative to the
opposite support hyperplane and the facet $F$.

\begin{theorem}\label{thm:facet-depth}
For every centered polytope $P$ and every $r\ge1$,
\[
 \pi_r(P)\le
 \min\left\{1,
 \sum_{F\text{ facet of }P}\theta_F^r\right\}.
\]
\end{theorem}

\begin{proof}
Fix collectively balanced independent vertices $V_1,\ldots,V_r$ and a facet $F$.  Let
$q_i=\Pr(V_i\in F)$.  Since $\ell_F=1$ on $F$ and $\ell_F\ge-a_F$ on $P$,
\[
 \mathbb E\ell_F(V_i)
 \ge q_i-a_F(1-q_i)
 =(1+a_F)q_i-a_F.
\]
Collective balance gives $\sum_i\mathbb E\ell_F(V_i)=0$, and therefore
\[
 \frac1r\sum_{i=1}^r q_i\le\frac{a_F}{1+a_F}=\theta_F.
\]
By independence and AM--GM,
\[
 \Pr(V_1,\ldots,V_r\in F)=\prod_{i=1}^r q_i\le\theta_F^r.
\]
The event that the selected vertices lie in a common facet is the union of these facet
events, so the union bound over the facets of $P$ proves the claim.
\end{proof}

\begin{corollary}[Centrally symmetric polytopes]\label{cor:cs-push}
If $P=-P$, then
\[
 \pi_r(P)\le
 \min\left\{1,\frac{f_{d-1}(P)}{2^r}\right\},
\]
where $f_{d-1}(P)$ is the number of facets of $P$.
\end{corollary}

\begin{proof}
For every normalized facet functional $\ell_F$, central symmetry gives
$\min_P\ell_F=-1$.  Hence $a_F=1$ and $\theta_F=1/2$ for every facet, and
Theorem~\ref{thm:facet-depth} applies.
\end{proof}

\subsection{The cube and the main theorem}

For $C_d=[-1,1]^d$, a tuple of vertices lies in a common facet exactly when there is a coordinate in
which all signs agree.  The general facet-depth estimate gives the upper bound, we show the lower bound below.

\begin{proof}[Proof of Theorem~\ref{thm:cube-push}]
The cube is centrally symmetric and has $2d$ facets, so
Corollary~\ref{cor:cs-push} gives
\[
 \pi_r(C_d)\le \min\left\{1,\frac d{2^{r-1}}\right\}.
\tag{4.1}
\]
It remains to show sharpness.  There are $2^{r-1}$ antipodal classes in
$\{-1,1\}^r/\{\pm1\}$.  Suppose first that $d\le2^{r-1}$ and choose $d$ distinct
antipodal classes, represented by columns
\[
 c^{(1)},\ldots,c^{(d)}\in\{-1,1\}^r.
\]
Let $a_i\in\{-1,1\}^d$ be the $i$th row of the resulting $r\times d$ sign matrix, and
let $V_i$ be uniform on $\{a_i,-a_i\}$.  Every $V_i$ has mean zero.  Write
$V_i=\eps_i a_i$ with independent uniform signs $\eps_i$.  In coordinate $t$, all
selected signs agree if and only if
\[
 (\eps_1,\ldots,\eps_r)=\pm c^{(t)}.
\]
Because the chosen columns represent distinct antipodal classes, these events are
disjoint, each with probability $2^{1-r}$.  Thus the pushability probability is exactly
$d/2^{r-1}$.

If $d\ge2^{r-1}$, use every antipodal class among the first $2^{r-1}$ coordinates and
repeat arbitrary columns thereafter.  Every sign vector $(\eps_1,\ldots,\eps_r)$ is then
$\pm c^{(t)}$ for some coordinate $t$, so a fixed coordinate occurs with probability
one.  This proves equality in (4.1).
\end{proof}

We now return to the orthogonal star problem that motivated the parameter.

\begin{proof}[Proof of Theorem~\ref{thm:c-main}]
Let $r=\lfloor\log_2 d\rfloor+2$, so $2^{r-1}>d$ and
$\pi_r(C_d)<1$ by Theorem~\ref{thm:cube-push}.  Suppose, toward a contradiction, that
none of $\F_1,\ldots,\F_r$ has an orthogonal star transversal.  Section~\ref{sec:benchmark}
associates to each color an open vertex indexed KKM cover
$\{A_v^i:v\in V(C_d)\}$.  Regard it as a polytopal KKMS cover by assigning the empty set
to every nonvertex face.  Prescribe $y_v^i=v$ at every vertex and choose arbitrary
representatives on nonvertex faces (which are irrelevant here).  Theorem~\ref{thm:pushed-kkms} gives labels
$v_1,\ldots,v_r$ and a common point
\[
 x\in\bigcap_{i=1}^r A_{v_i}^i
\]
such that the $v_i$ are not contained in a common facet.  Choose witnesses
$F_i\in\F_i$ with $F_i\subset R_x^{v_i}$.  By hypothesis the rainbow choice
$F_1,\ldots,F_r$ has a common axis-parallel hyperplane transversal.  The pushing lemma (Lemma \ref{lem:pushing})
then places all $v_i$ in a common facet, a contradiction.
\end{proof}

\section{Separated $d$-intervals and lower constructions}\label{sec:dintervals}

\subsection{The multitrack interpretation}

By the transversal dictionary, Proposition~\ref{prop:dictionary}, the box formulation is
equivalent to the language of full separated $d$-intervals: a common axis-parallel
hyperplane is translated to strong intersection on one track, and an orthogonal star is a weak
transversal with one point on each track.  Thus Theorem~\ref{thm:c-main} immediately
becomes the following statement.

\begin{corollary}\label{cor:dintervals}
Let $ r=\left\lfloor\log_2d\right\rfloor+2$ and let $\F_1,\ldots,\F_r$ be finite families of full separated $d$-intervals.  Suppose
that for every rainbow choice $I_i\in\F_i$, we have that $I_1,\ldots,I_r$ are strongly intersecting.  Then for some $j$, $\F_j$ is weakly intersecting.
\end{corollary}

\subsection{A pairwise obstruction with $2d$ members}\label{sec:lower}

The following explicit decomposition is useful for seeing why pairwise transversality is
not enough.

\begin{lemma}\label{lem:ariadna-decomp}
For every $d\ge3$, the edges of $K_{2d}$ can be partitioned into interval graphs
$G_1,\ldots,G_d$ on the common vertex set
\[
 X\sqcup Y,
 \qquad
 X=\{v_1,\ldots,v_d\},
 \quad
 Y=\{v_{d+1},\ldots,v_{2d}\},
\]
such that
\[
 G_1=K_X\sqcup K_Y
\]
and every edge of $G_i$, $i\ge2$, joins $X$ to $Y$.  Moreover, if
$Q_i$ is a clique of $G_i$ for each $i\in[d]$, then
\[
 Q_1\cup\cdots\cup Q_d\ne V(K_{2d}).
\]
In other words, the vertices cannot be covered by choosing one clique from each of the
$d$ interval graphs.
\end{lemma}

We give the decomposition explicitly.  The graphs $G_2,\ldots,G_{d-1}$ will be paths,
while $G_d$ is a small linear forest; thus all of them are interval graphs, as is
$G_1=K_X\sqcup K_Y$.  We then verify separately that their edge sets are disjoint and
exhaust $E(K_{2d})$, and finally prove the clique-cover obstruction.

\begin{proof}
Write
$V(K_{2d})=\{v_1,\ldots,v_{2d}\}$ and put
\[
 X=\{v_1,\ldots,v_d\},
 \qquad
 Y=\{v_{d+1},\ldots,v_{2d}\}.
\]
Let
\[
 G_1=K_X\sqcup K_Y.
\]
For $2\le i\le d-1$, define a vertex sequence
$u_1,\ldots,u_{2d-2i+4}$ by
\[
 u_j=
 \begin{cases}
 v_{i+(j-3)/2},& i\text{ even and }j\text{ odd},\\
 v_{2d-(j-2)/2},& i\text{ even and }j\text{ even},\\
 v_{2d-i+(5-j)/2},& i\text{ odd and }j\text{ odd},\\
 v_{j/2},& i\text{ odd and }j\text{ even}.
 \end{cases}
\]
and let $G_i$ be the path with this vertex sequence.  Finally, set
\[
 E(G_d)=
 \begin{cases}
 \{v_{d-1}v_{2d},\ v_{2d}v_d,\ v_dv_{2d-1},\ v_1v_{d+1}\},&d\text{ even},\\
 \{v_{d+2}v_1,\ v_1v_{d+1},\ v_{d+1}v_2,\ v_{2d}v_d\},&d\text{ odd}.
 \end{cases}
\tag{5.1}
\]
Thus $G_d$ is the disjoint union of a path and one edge (see Figure \ref{fig:examplec6} for an example for $d=6$).  Paths, linear forests,
and disjoint unions of cliques are interval graphs; see, for example,
\cite{Lekkerkerker}.

\begin{figure}
    \centering
    \includegraphics[width=0.9\linewidth]{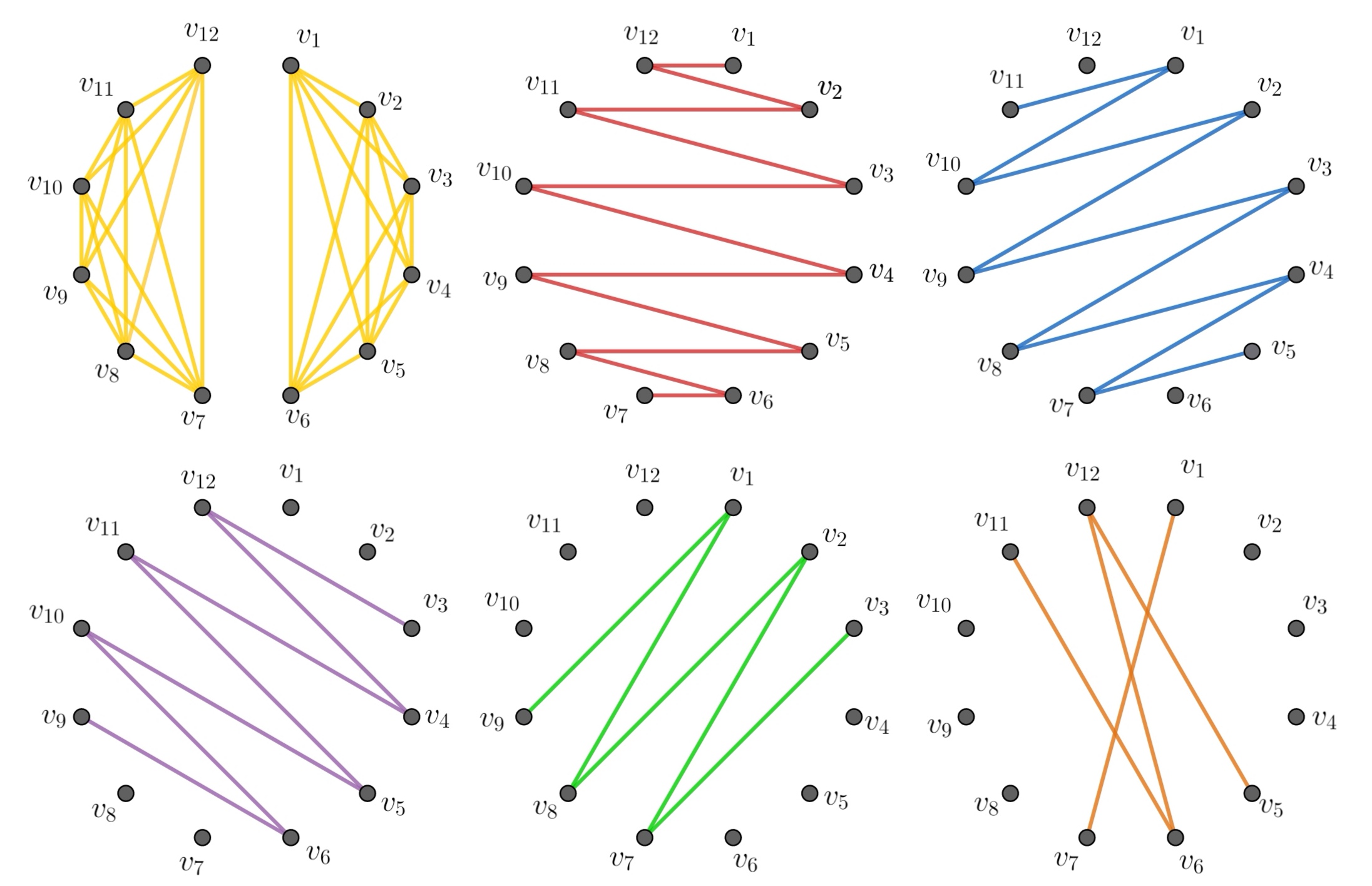}
    \caption{Example that shows $c_6\geq 3$}
    \label{fig:examplec6}
\end{figure}

Every edge of $G_i$, $i\ge2$, joins $X$ to $Y$.  For the disjointness verification,
observe that for $2\le i\le d-1$ the sum of the indices of the endpoints of every edge of
$G_i$ belongs to
\[
 B_i=
 \begin{cases}
 \{2d+i-1,2d+i\},&i\text{ even},\\
 \{2d-i+2,2d-i+3\},&i\text{ odd}.
 \end{cases}
\]
For the final graph take
\[
 B_d=
 \begin{cases}
 \{d+2,3d-1,3d\},&d\text{ even},\\
 \{d+2,d+3,3d\},&d\text{ odd}.
 \end{cases}
\]
These sets are pairwise disjoint, and the edges in (5.1) are also disjoint from the
preceding paths.  Since $G_1$ contains exactly the edges internal to $X$ or internal to $Y$,
it is disjoint from every $G_i$ with $i\ge2$.

It remains only to count:
\[
\begin{aligned}
 |E(G_1)|&=d(d-1),\\
 |E(G_i)|&=2d-2i+3\quad(2\le i\le d-1),\\
 |E(G_d)|&=4.
\end{aligned}
\]
Hence
\[
 \sum_{i=1}^d|E(G_i)|=\binom{2d}{2},
\]
so the $G_i$ partition $E(K_{2d})$.

Finally, every clique of
$G_1=K_X\sqcup K_Y$ is contained entirely in $X$ or entirely in $Y$; suppose, by
symmetry, that $Q_1\subseteq X$.  For every $i\ge2$ there is no edge of $G_i$ inside
$Y$, so a clique $Q_i$ contains at most one vertex of $Y$.  Thus
$Q_2\cup\cdots\cup Q_d$ covers at most $d-1$ of the $d$ vertices of $Y$, and at
least one vertex of $Y$ is missed by $Q_1\cup\cdots\cup Q_d$.  This proves the
lemma.
\end{proof}

By Proposition~\ref{prop:dictionary} we conclude the following.

\begin{proposition}\label{prop:pairwise-obstruction}
For every $d\ge3$ there is a family
\[
 \F=\{F_1,\ldots,F_{2d}\}
\]
of axis-parallel boxes in $\R^d$ such that every two members have a common axis-parallel
hyperplane transversal, but no orthogonal star is transversal to all of $\F$.
\end{proposition}

Taking all color classes equal to $\F$ gives
\[
 h_d\ge3,
 \qquad
 c_d\ge3
 \qquad(d\ge3).
\]

Combining the lower bound with Theorems~\ref{thm:h-benchmark} and \ref{thm:c-main} gives
\[
 h_3=h_4=h_5=h_6=3,
 \qquad
 c_3=3.
\]

For completeness we also record the planar counterexample for two colors.

\begin{proposition}\label{prop:two-colors-plane}
There are two families $\F_1,\F_2$ of axis-parallel rectangles in $\R^2$ such that every
cross pair $F_1\in\F_1$, $F_2\in\F_2$ has a common axis-parallel line transversal, but
neither color class has an orthogonal star transversal.  Hence $c_2=3$.
\end{proposition}

\begin{proof}
Take
\[
\begin{aligned}
\F_1=\{&[5,11]\times[21,27],\ [13,19]\times[13,19],\ [21,27]\times[5,11]\},\\
\F_2=\{&[1,7]\times[9,15],\ [9,15]\times[1,7],\ [17,23]\times[17,23]\}.
\end{aligned}
\]

\begin{figure}[tbp]
\centering
\begin{tikzpicture}[x=0.28cm,y=0.28cm]
  % Family F_1
  \filldraw[draw=blue!70!black,fill=blue!50,fill opacity=0.35] (5,21) rectangle (11,27);
  \filldraw[draw=blue!70!black,fill=blue!50,fill opacity=0.35] (13,13) rectangle (19,19);
  \filldraw[draw=blue!70!black,fill=blue!50,fill opacity=0.35] (21,5) rectangle (27,11);

  % Family F_2
  \filldraw[draw=red!70!black,fill=red!50,fill opacity=0.35] (1,9) rectangle (7,15);
  \filldraw[draw=red!70!black,fill=red!50,fill opacity=0.35] (9,1) rectangle (15,7);
  \filldraw[draw=red!70!black,fill=red!50,fill opacity=0.35] (17,17) rectangle (23,23);

  % Labels
  \node at (8,24) {$1,1$};
  \node at (16,16) {$1,2$};
  \node at (24,8) {$1,3$};

  \node at (4,12) {$2,1$};
  \node at (12,4) {$2,2$};
  \node at (20,20) {$2,3$};
\end{tikzpicture}
\caption{Counterexample for two colors on the plane.}
\label{fig:countplane}
\end{figure}
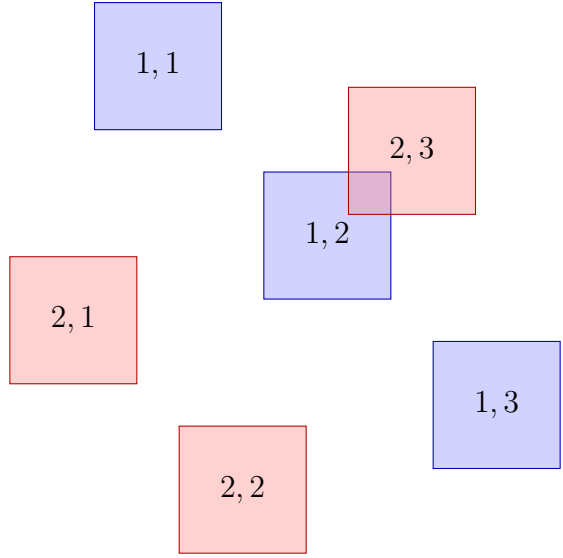

The example is depicted in Figure~\ref{fig:countplane}. A direct check shows that every bicolored  pair overlaps in at least one coordinate projection,
so it has a common vertical or horizontal transversal.  Within either color class, the
three $x$-projections are pairwise disjoint and the three $y$-projections are pairwise
disjoint.  Hence one vertical and one horizontal line can meet at most two of the three
rectangles.
\end{proof}

\section{Actual colorful Kaiser--Rabinovich}\label{sec:actual-KR}

The main development of this paper concerns the case in which each component is an
interval.  This is where the cubical geometry is visible and where rainbow pushability can
be computed exactly.  We now return to the general setting in Kaiser and
Rabinovich's original paper \cite{KaiserRabinovich1999}.

\begin{definition}
An $(n,D)$-body is an ordered $n$-tuple
\[
 K=(K_1,\ldots,K_n)
\]
of convex subsets of $\R^D$; empty components are allowed.  A family $\mathcal G$ of
$(n,D)$-bodies is \emph{strongly intersecting} if for some $i\in[n]$,
\[
 \bigcap_{K\in\mathcal G}K_i\ne\varnothing,
\]
and is \emph{weakly intersecting} if there are points $p_1,\ldots,p_n\in\R^D$ such that
every $K\in\mathcal G$ contains $p_i$ in at least one component $K_i$.
\end{definition}

For integers $m,k\ge1$, write
\[
 \mathsf C(n,D;m,k)
\]
for the following assertion: whenever $\F_1,\ldots,\F_m$ are finite nonempty color
classes of convex $(n,D)$-bodies and every rainbow subfamily of size at most $k$ is
strongly intersecting, some color class is weakly intersecting.  When $k=m$, the
hypothesis is equivalent to requiring only complete rainbow choices, since strong
intersection is inherited by subfamilies.

Kaiser and Rabinovich denote by $H(n,D)$ the smallest integer $h$ such that every finite
family of convex $(n,D)$-bodies whose every $h$-member subfamily is strongly intersecting
is weakly intersecting.  They prove the submultiplicativity
\begin{equation}
 H(n,D_1+D_2)\le H(n,D_1)H(n,D_2).
\tag{7.1}\label{eq:KR-submult}
\end{equation}
Their proof is a projection/fiber argument, and the same organization is compatible with
colors.

\begin{proposition}\label{prop:colorful-product}
If
\[
 \mathsf C(n,D_1;m_1,k_1)
 \qquad\text{and}\qquad
 \mathsf C(n,D_2;m_2,k_2)
\]
hold, then
\[
 \mathsf C(n,D_1+D_2;m_1m_2,k_1k_2)
\]
holds.
\end{proposition}

\begin{proof}
Identify every component space as
\[
 \R^{D_1+D_2}=\R^{D_1}\times\R^{D_2}.
\]
The argument has two stages.  We first project intersections of small rainbow families to
the first factor and use the $D_1$-dimensional statement to choose one good block of
colors together with one fiber in each component.  We then restrict that block to those
fibers and apply the $D_2$-dimensional statement.

Partition the $m_1m_2$ original color classes into $m_1$ blocks
$B_1,\ldots,B_{m_1}$ of $m_2$ colors each.  Let $\operatorname{pr}$ denote projection
onto $\R^{D_1}$.  For a block $B_a$ and a nonempty rainbow subfamily $\mathcal G$ using
colors from $B_a$, with $|\mathcal G|\le k_2$, form the convex $(n,D_1)$-body
\[
 Q_{\mathcal G}:=
 \left(
 \operatorname{pr}\!\left(\bigcap_{K\in\mathcal G}K_1\right),\ldots,
 \operatorname{pr}\!\left(\bigcap_{K\in\mathcal G}K_n\right)
 \right).
\]
For each $a$, regard all such $Q_{\mathcal G}$ as one meta-color class.

Take at most $k_1$ meta-bodies from distinct meta-colors, say
$Q_{\mathcal G_1},\ldots,Q_{\mathcal G_s}$.  Since the blocks are disjoint,
$\mathcal G_1\cup\cdots\cup\mathcal G_s$ is a rainbow original subfamily of size at most
$k_1k_2$, hence is strongly intersecting.  Projecting a common point in the witnessing
component shows that the selected meta-bodies are strongly intersecting.  Thus
$\mathsf C(n,D_1;m_1,k_1)$ yields a block $B_*$ whose entire meta-color class is weakly
intersecting.  Choose points $p_i\in\R^{D_1}$ witnessing this.

For an original body $K$ whose color lies in $B_*$, define a convex $(n,D_2)$-body by
restricting to the corresponding fibers:
\[
 K^{p}:=\bigl(K_1^{p_1},\ldots,K_n^{p_n}\bigr),
 \qquad
 K_i^{p_i}:=\{y\in\R^{D_2}:(p_i,y)\in K_i\}.
\]
For every rainbow $\mathcal G\subseteq B_*$ of size at most $k_2$, weak intersection of
the meta-color implies that for some $i$,
\[
 p_i\in\operatorname{pr}\!\left(\bigcap_{K\in\mathcal G}K_i\right).
\]
Hence there is a common $y\in\R^{D_2}$ with $(p_i,y)\in K_i$ for every
$K\in\mathcal G$, so the restricted family $\{K^p:K\in\mathcal G\}$ is strongly
intersecting.  Applying $\mathsf C(n,D_2;m_2,k_2)$ inside $B_*$ gives an original color
class, call it $\F_*$, together with points $q_1,\ldots,q_n\in\R^{D_2}$ witnessing weak
intersection of the restricted bodies.  Thus for every $K\in\F_*$ there is an index
$i\in[n]$ such that $q_i\in K_i^{p_i}$, equivalently $(p_i,q_i)\in K_i$.  Hence
\[
 (p_1,q_1),\ldots,(p_n,q_n)\in\R^{D_1+D_2}
\]
witness that $\F_*$ is weakly intersecting in the original space.
\end{proof}

To use Proposition~\ref{prop:colorful-product}, we need the one-dimensional statements
for general $n$-intervals, where some components may be empty.  The proofs already
developed for full separated intervals extend with only a harmless change in the cubical
cover.

\begin{lemma}\label{lem:KR-1d-inputs}
With $L_n$ and $R_n$ as defined in the introduction,
\[
 \mathsf C(n,1;n+1,L_n)
 \qquad\text{and}\qquad
 \mathsf C(n,1;R_n,R_n)
\]
hold.
\end{lemma}

\begin{proof}
By the finite interval reduction of Appendix~\ref{app:technical}, we may assume that all
nonempty components are compact and, after an affine rescaling on each track, contained
in $[-1,1]$.  Thus only empty components require a change in the cubical cover.  Let
$x\in C_n$ and $v\in\{-1,1\}^n$.  For an
$n$-interval $I=(I_1,\ldots,I_n)$ missed by $x$, say that $I$ is \emph{compatible} with $v$ at
$x$ if, for every nonempty component $I_t$,
\[
 (x_t-y)v_t>0\qquad(y\in I_t);
\]
empty components impose no condition on $v_t$.  Define the cubical KKM sets using this
compatibility relation.  If a point $x$ lies in a face $G$ of the cube and misses an
$n$-interval, one can choose a vertex $v\in V(G)$ compatible with it: the signs on
nonempty components are forced, while empty components may be chosen freely.  Hence the
KKM covering condition remains valid.

Likewise, if selected $n$-intervals are strongly intersecting in component $t$, then all
their $t$-th components are nonempty and contain a common point.  The same calculation as
in Lemma~\ref{lem:pushing} forces all compatible cube labels to have the same $t$-th sign.
Thus the cubical pushing lemma (Lemma \ref{lem:pushing}) remains valid.  The benchmark proof of
Theorem~\ref{thm:h-benchmark} and the pushability proof of Theorem~\ref{thm:c-main}
therefore apply verbatim, with $d$ replaced by $n$.
\end{proof}

\begin{proof}[Proof of Theorem~\ref{thm:KR-tradeoffs}]
Iterating Proposition~\ref{prop:colorful-product} $D-1$ times with the first statement of
Lemma~\ref{lem:KR-1d-inputs} gives
\[
 \mathsf C(n,D;(n+1)^D,L_n^D),
\]
which is part (i).  Iterating with the second statement gives
\[
 \mathsf C(n,D;R_n^D,R_n^D),
\]
and since here the test size equals the number of colors, this is exactly the
complete-rainbow statement in part (ii).
\end{proof}

\begin{remark}
The same proof allows the two one-dimensional inputs to be mixed.  For every
$0\le t\le D$ it gives
\[
 \mathsf C\bigl(n,D;(n+1)^tR_n^{D-t},L_n^tR_n^{D-t}\bigr).
\]
\end{remark}

There is also a reason not to seek a literal higher-dimensional version of the cubical
pushability parameter.  Kaiser and Rabinovich already point out in their Remark~6.3 that
a straightforward vector-valued analogue of the balancing lemma behind the $D=1$
argument cannot recover their polylogarithmic bound: in dimension $2$ they construct a
balanced family for which any subfamily surrounding the origin in every component must
have at least $n$ members \cite{KaiserRabinovich1999}.

\section{Open problems}\label{sec:final}

Beyond the small-dimensional values settled in Section~\ref{sec:dintervals}, our lower
bounds remain essentially constant.  Thus even the qualitative questions
\[
 \text{does }h_d\to\infty?,\qquad \text{does }c_d\to\infty?
\]
are open.  More ambitiously, it is natural to ask whether either parameter has logarithmic
order of growth, as suggested by the upper bounds.

The parameter $c_d$ may be a more approachable first target for a nonconstant lower
bound.  To prove $c_d>r$ one may use $r$ genuinely different color classes and needs to
arrange the transversal condition only for complete rainbow choices.  A lower
bound for $h_d$, by contrast, must control all rainbow subfamilies up to a prescribed
size among $d+1$ colors.  At present we do not know an example which requires $c_d>3$ in any
dimension.

The first unresolved case is already
\[
 c_4\in\{3,4\}.
\]
The pointwise pushability method cannot decide it: Theorem~\ref{thm:cube-push} gives
$\pi_3(C_4)=1$.  Thus proving $c_4=3$ would require information beyond the balanced distributions at one point.

A second family of problems concerns rainbow pushability itself.  It would be useful to
compute $\pi_r(P)$ exactly for further centered polytopes and to understand when the
facet-depth bound of Theorem~\ref{thm:facet-depth} is sharp.  The cube shows that the union bound is sometimes be exact.

The full Kaiser--Rabinovich setting adds a third direction.  Theorem~\ref{thm:KR-tradeoffs}
uses $(n+1)^D$ colors to retain the original local threshold $L_n^D$.  It is natural to
ask how far this color count can be reduced. At the opposite endpoint, the complete-rainbow theorem uses
$R_n^D$ colors. Determining the correct order, or even whether the product bound is close
to optimal for fixed $D>1$, is widely open.

\appendix

\section{Topological and compactness details}\label{app:technical}

\subsection{The vertexwise Sperner--Shapley step}

We give the short verification of Lemma~\ref{lem:face-preserving-selection}.  The
one-dimensional case is immediate, so assume $\dim P\ge2$.  Extend $u\mapsto z(u)$
affinely over every simplex of $T$, obtaining a continuous map $Z:P\to P$.  If $G$ is a
face of $P$, then every triangulation vertex in $G$ has support contained in $G$, and
therefore $Z(G)\subseteq G$.  On the boundary, the straight-line homotopy between $Z$
and the identity remains inside the face containing each point.  Consequently
$Z|_{\partial P}$ has degree one, exactly as in the proof of
\cite[Theorem~2.2]{FrickZerbib2019}.

If an interior point $p$ were omitted by $Z$, radial projection from $p$ onto $\partial P$
would turn $Z$ into an extension over $P$ of this degree-one boundary map.  That is
impossible, since a map of the boundary sphere which extends over the ball has degree
zero.  Hence every interior point lies in the image of $Z$; since $Z(P)$ is compact, it
follows that $Z$ is surjective.  If $Z(x)=p$ and $x$ lies in a simplex $\tau$, then
affinity on $\tau$ gives
\[
 p=Z(x)\in\conv\{z(u):u\in V(\tau)\},
\]
as required.

\subsection{Finite interval reduction}\label{app:finite-reduction}

The finite-family results do not require bounded or closed sets.  We record the elementary
reduction in the slightly more general multitrack form needed in Section~\ref{sec:actual-KR}.
Suppose finitely many bodies have interval components, possibly empty, open, or unbounded,
and suppose a finite collection of subfamilies is required to be strongly intersecting.
For each required strong intersection, choose one witnessing point in one witnessing
component.  Inside every nonempty interval component, retain the convex hull of all chosen
witnesses lying in that component, adding one arbitrary point of the component if none was
chosen.  The resulting nonempty components are compact intervals contained in the original
ones.  Every prescribed strong intersection is preserved, while any weak transversal to
the reduced bodies is also a weak transversal to the original bodies.

Applied to the coordinate projections of a finite family of nonempty connected subsets of
$\R^d$, this shows that the finite versions of Theorems~\ref{thm:h-benchmark} and
\ref{thm:c-main} remain valid without compactness, boundedness, or closedness assumptions.
Indeed, connectedness makes every coordinate projection an interval, and the reduced
intervals define compact boxes to which the proved statements apply.

\subsection{Infinite compact families}\label{app:infinite}

We finally record the common compactness argument for the two orthogonal star theorems.
Let $m=d+1$ for Theorem~\ref{thm:h-benchmark}, or
$m=\lfloor\log_2d\rfloor+2$ for Theorem~\ref{thm:c-main}, and let
$\F_1,\ldots,\F_m$ be arbitrary, possibly infinite, families of compact connected sets
satisfying the corresponding rainbow hypothesis.  Put
$\overline\R=\R\cup\{-\infty,+\infty\}$ and $\operatorname{pr}_i$ as the projection to the $i$-coordinate.  For a compact
set $F\subset\R^d$ define
\[
 C_F:=\{x\in\overline\R^{\,d}:x_i\in\operatorname{pr}_i(F)\text{ for some }i\in[d]\}.
\]
This is closed.  For every finite
$\mathcal G\subseteq\bigcup_{j=1}^m\F_j$, put
\[
 K_{\mathcal G}:=
 \bigcup_{j=1}^m
 \left(\bigcap_{F\in\mathcal G\cap\F_j}C_F\right)\times\{j\}
 \subseteq \overline\R^{\,d}\times[m],
\]
with an empty intersection interpreted as all of $\overline\R^{\,d}$.  If some color is
absent from $\mathcal G$, then $K_{\mathcal G}$ is automatically nonempty.  Otherwise the
corresponding finite theorem applied to the finite restrictions gives a color whose
restricted class has an orthogonal star transversal, and again $K_{\mathcal G}\ne\varnothing$.
Moreover,
\[
 K_{\mathcal G_1}\cap\cdots\cap K_{\mathcal G_s}
 =K_{\mathcal G_1\cup\cdots\cup\mathcal G_s}.
\]
Thus these closed sets have the finite-intersection property.  Compactness gives
$(z,j^*)$ belonging to all of them.  Taking $\mathcal G=\{F\}$ shows that for every
$F\in\F_{j^*}$ some finite coordinate $z_i$ lies in $\operatorname{pr}_i(F)$.  Replacing any infinite
coordinates of $z$ by arbitrary real numbers produces an orthogonal star meeting every
member of $\F_{j^*}$.

\section{The transversal dictionary}\label{app:dictionary}

We prove Proposition~\ref{prop:dictionary}.  Since each $F_j$ is connected, each coordinate
projection $I_i^j=\operatorname{pr}_i(F_j)$ is an interval.  For a coordinate hyperplane
$H_i(a)=\{x:x_i=a\}$ we have
\[
 H_i(a)\cap F_j\ne\varnothing
 \quad\Longleftrightarrow\quad
 a\in I_i^j
 \quad\Longleftrightarrow\quad
 H_i(a)\cap B_j\ne\varnothing.
\]
Hence a subfamily indexed by $S$ has a common axis-parallel hyperplane exactly when
$\bigcap_{j\in S}I_i^j\ne\varnothing$ for some $i$.  Since $S$ is finite, the Helly
property for intervals says that this is equivalent to the intervals
$\{I_i^j:j\in S\}$ being pairwise intersecting, that is, to $S$ being a clique of $G_i$.
This proves the first chain of equivalences.

Likewise, an orthogonal star $H_1(x_1)\cup\cdots\cup H_d(x_d)$ meets $F_j$ exactly when
$x_i\in I_i^j$ for at least one $i$.  Thus such a star exists precisely when
\[
 J=\bigcup_{i=1}^d Q_i,
 \qquad
 Q_i:=\{j:x_i\in I_i^j\}.
\]
Each $Q_i$ is a clique of $G_i$.  Conversely, if cliques $Q_i$ cover $J$, then the finite
Helly property for intervals gives a point $x_i\in\bigcap_{j\in Q_i}I_i^j$ for every
nonempty $Q_i$; choose $x_i$ arbitrarily when $Q_i=\varnothing$.  The resulting
orthogonal star meets every $F_j$.  This proves the second chain and completes the
dictionary.

\section*{Acknowledgements}
The first author was supported by UNAM-PAPIIT project IN119026.  He also acknowledges
support from the UNAM DGAPA-PASPA program for a sabbatical stay at the Instituto de
Matem\'aticas, Unidad Juriquilla, UNAM.  The second author acknowledges support from a
SECIHTI graduate fellowship with CVU 1233555.

The research leading to this work involved an extended interaction with ChatGPT (OpenAI), using GPT-5.6 Sol. The model was used throughout the project for mathematical exploration, development and refinement of proof strategies, preliminary drafting, and computational work. The arguments, formulations, and computations in the paper emerged through repeated collaborative interaction, revision, and correction during this process. The final proofs, computations, and references have been checked by the authors, who assume full responsibility for the mathematical content of the paper, including any errors or bibliographic omissions.

\printbibliography

\end{document}